\documentclass[12pt]{amsart}

\usepackage{a4wide}
\usepackage{enumerate}

\newtheorem{theorem}{Theorem}[section]
\newtheorem{lemma}[theorem]{Lemma}
\newtheorem{proposition}[theorem]{Proposition}
\newtheorem{corollary}[theorem]{Corollary}

\theoremstyle{remark}
\newtheorem{remark}[theorem]{Remark}

\newcommand{\C}{\ensuremath{\mathbb{C}}}
\newcommand{\R}{\ensuremath{\mathbb{R}}}
\renewcommand{\H}{\ensuremath{\mathbb{H}}}
\newcommand{\g}[1]{\ensuremath{\mathfrak{#1}}}

\newcommand{\mi}{\ensuremath{m_{0}}}
\newcommand{\mii}{\ensuremath{m_{\pi/2}}}

\DeclareMathOperator{\sech}{sech}
\DeclareMathOperator{\csch}{csch}
\DeclareMathOperator{\codim}{codim}
\DeclareMathOperator{\id}{I}

\begin{document}
\title[Isoparametric hypersurfaces]{Isoparametric
hypersurfaces in Damek-Ricci spaces}

\author[J.C. D\'\i{}az-Ramos]{Jos\'{e} Carlos D\'\i{}az-Ramos}
\address{Department of Geometry and Topology,
University of Santiago de Compostela, Spain.}
\email{josecarlos.diaz@usc.es}

\author[M. Dom\'{\i}nguez-V\'{a}zquez]{Miguel Dom\'{\i}nguez-V\'{a}zquez}
\address{Department of Geometry and Topology,
University of Santiago de Compostela, Spain.}
\email{miguel.dominguez@usc.es}

\thanks{The first author has been supported by a Marie-Curie
European Reintegration Grant (PERG04-GA-2008-239162). The second
author has been supported by the FPU programme of the Spanish
Government. Both authors have been supported by projects
MTM2009-07756 and INCITE09207151PR (Spain).}



\begin{abstract}
We construct uncountably many isoparametric families of hypersurfaces
in Damek-Ricci spaces. We characterize those of them that have
constant principal curvatures by means of the new concept of
generalized K\"ahler angle. It follows that, in general, these
examples are inhomogeneous and have nonconstant principal curvatures.

We also find new cohomogeneity one actions on quaternionic hyperbolic
spaces, and an isoparametric family of inhomogeneous hypersurfaces
with constant principal curvatures in the Cayley hyperbolic plane.
\end{abstract}


\maketitle

\section{Introduction}

A connected hypersurface of a Riemannian manifold is called an
isoparametric hypersurface if its nearby parallel hypersurfaces have
constant mean curvature. Cartan characterized isoparametric
hypersurfaces in real space forms as those hypersurfaces with
constant principal curvatures, and achieved the classification of
these objects in real hyperbolic spaces. Segre found the analogous
classification for Euclidean spaces. In both cases, every
isoparametric hypersurface is an open part of an (extrinsically)
homogeneous one, that is, an open part of an orbit of a cohomogeneity
one action.

Nevertheless, the problem in spheres turned out to be much more
involved and rich. Although the classification of homogeneous
hypersurfaces in spheres is known, it is remarkable that not every
complete isoparametric hypersurface in a sphere is homogeneous. The
known inhomogeneous examples were constructed by Ferus, Karcher and
M\"unzner in \cite{FKM81}. Recently, much progress has been made for
spheres~\cite{CCJ07,Ch11,Ch,I08,Mi}. A complete classification has
not been achieved yet, but there is only one unsettled case. For a
survey on these problems and other related topics, we refer to
\cite{T00} and~\cite{T10}.

In more general ambient spaces of nonconstant curvature, the
equivalence between isoparametric hypersurfaces and hypersurfaces
with constant principal curvatures is no longer true. The first
examples were found by Wang \cite{W82}, who constructed some
inhomogeneous isoparametric hypersurfaces with nonconstant principal
curvatures in the complex projective space, by projecting some of the
inhomogeneous hypersurfaces in spheres via the Hopf map.

Recently, the authors have found a large set of inhomogeneous
isoparametric hypersurfaces with nonconstant principal curvatures in
complex hyperbolic spaces \cite{DRDV10}. To our knowledge, these are
the first examples of inhomogeneous isoparametric hypersurfaces in
Riemannian symmetric spaces whose construction does not depend on the
examples in spheres.

The aim of this article is to provide a construction method of
isoparametric hypersurfaces in Damek-Ricci harmonic spaces. These
homogeneous spaces are a family that contains the rank one noncompact
symmetric spaces as particular cases. They were constructed by Damek
and Ricci in \cite{DR92} and they provide counterexamples to the
so-called Lichnerowicz conjecture, stating that every Riemannian
harmonic manifold is locally isometric to a two-point homogeneous
space. The hypersurfaces that we introduce arise as tubes around
certain homogeneous minimal submanifolds whose construction extends
the one proposed by Berndt and Br\"uck~\cite{BB01}, and by the
authors~\cite{DRDV10}. Although the definition of the new examples is
relatively straightforward, the verification that these examples have
the desired properties is far from being trivial.

The main concept introduced in this paper is that of
\emph{generalized K\"{a}hler angle}, which generalizes previous notions
of K\"{a}hler angle and quaternionic K\"{a}hler angle~\cite{BB01}. Among the
isoparametric hypersurfaces we construct in this paper, the ones with
constant principal curvatures are precisely those whose focal
submanifolds have normal spaces of constant generalized K\"ahler
angle (Theorem~\ref{thMain}). As a consequence, we obtain uncountably
many noncongruent isoparametric families of inhomogeneous
hypersurfaces with nonconstant principal curvatures in complex and
quaternionic hyperbolic spaces. Compare this with the case of
spheres, where the known set of inhomogeneous isoparametric families
is countable~\cite{FKM81}.

For the quaternionic hyperbolic spaces, we also provide new examples
of cohomogeneity one actions. Recall that quaternionic hyperbolic
spaces are the unique Riemannian symmetric spaces of rank one for
which a classification of cohomogeneity one actions is still not
known.

In the Cayley hyperbolic plane, our method yields an inhomogeneous
isoparametric family of hypersurfaces with constant principal
curvatures, which is, to our knowledge, the first such example in a
Riemmanian symmetric space different from a sphere
(cf.~\cite[p.~7]{GTY}).

This article is organized as follows. In
Section~\ref{secPreliminaries} we set up the fundamental definitions
and results on Damek-Ricci spaces. The definition of generalized
K\"{a}hler angle is presented in Section~\ref{secGKA}. In
Section~\ref{sec:Examples} the new examples of isoparametric
hypersurfaces in Damek-Ricci harmonic spaces are introduced. We start
by defining the focal set of the new examples in \S\ref{secFocal},
and then in \S\ref{secJacobi} we investigate the properties of the
tubes around these submanifolds using Jacobi field theory. The main
result of this work is stated in Theorem~\ref{thMain}. Finally, in
Section~\ref{secSymmetric} we consider some particular cases in the
rank one symmetric spaces of noncompact type. In \S\ref{sec:HHn} we
construct new examples of cohomogeneity one actions on quaternionic
hyperbolic spaces (Theorem~\ref{th:cohom1}), and in \S\ref{sec:OH2}
we give an example of an inhomogeneous isoparametric hypersurface in
the Cayley hyperbolic plane (Theorem~\ref{th:CH2}).

The authors would like to thank Prof.\ J\"{u}rgen Berndt for reading a
draft version of this paper and suggesting a proof of
Theorem~\ref{th:cohom1}.

\section{Generalized Heisenberg groups and Damek-Ricci spaces}\label{secPreliminaries}

In this section we recall the construction of Damek-Ricci spaces,
presenting some of the properties that we will use later. Since the
description of such spaces depends on the so-called generalized
Heisenberg algebras, we begin by defining these structures. The main
reference for all these notions is \cite{BTV95}, where one can find
the proofs of the results presented below, as well as further
information on Damek-Ricci spaces.

\subsection{Generalized Heisenberg algebras and
groups}\label{sec:Heisenberg}
Let $\g{v}$ and $\g{z}$ be real vector
spaces and $\beta:\g{v}\times\g{v}\to\g{z}$ a skew-symmetric bilinear
map. Define the direct sum $\g{n}=\g{v}\oplus\g{z}$ and endow it with
an inner product $\langle\cdot, \cdot\rangle_\g{n}$ such that $\g{v}$
and $\g{z}$ are perpendicular. Define a linear map $J\colon
Z\in\g{z}\mapsto J_Z\in\textup{End}(\g{v})$ by
\[
\langle J_Z U,V\rangle=\langle\beta(U,V),Z\rangle,\quad
\text{ for all } U, V\in\g{v}, \; Z\in\g{z},
\]
and a Lie algebra structure on $\g{n}$ by
\[
[U+X,V+Y]=\beta(U,V),\quad \text{ for all } U,V\in\g{v}, \; X,Y\in\g{z},
\]
or equivalently, by
\[
\langle[U,V],X\rangle=\langle J_X U, V\rangle,
\quad [X,V]=[U,Y]=[X,Y]=0,\quad
\text{ for all } U,V\in\g{v}, \; X,Y\in\g{z}.
\]
Then, $\g{n}$ is a two-step nilpotent Lie algebra with center
$\g{z}$, and, if $J_Z^2=-\langle Z,Z\rangle \id_\g{v}$ for all
$Z\in\g{z}$, $\g{n}$ is said to be a \emph{generalized Heisenberg
algebra} or an \emph{H-type algebra}. (Here and henceforth the
identity is denoted by $\id$.) The associated simply connected
nilpotent Lie group $N$, endowed with the induced left-invariant
Riemannian metric, is called a \emph{generalized Heisenberg group} or
an \emph{H-type group}.

Let $U$, $V\in\g{v}$ and $X$, $Y\in\g{z}$. In this work, we will make
use of the following properties of generalized Heisenberg algebras
without explicitly referring to them:
\begin{align*}
J_X J_Y+J_Y J_X&=-2\langle X,Y\rangle \id_\g{v},&
[J_X U, V]-[U, J_X V]&=-2\langle U,V\rangle X,\\
\langle J_X U, J_X V\rangle&=\langle X,X\rangle\langle U, V\rangle,&
\langle J_X U, J_Y U\rangle&=\langle X,Y\rangle\langle U, U\rangle.
\end{align*}
In particular, for any unit $Z\in\g{z}$, $J_Z$ is an almost Hermitian
structure on $\g{v}$.

The map $J\colon\g{z}\to\textup{End}(\g{v})$ can be extended to the
Clifford algebra $Cl(\g{z},q)$, where $q$ is the quadratic form given
by $q(Z)=-\langle Z,Z\rangle$, in such a way that $\g{v}$ becomes now
a Clifford module over $Cl(\g{z},q)$ (see \cite[Chapter 3]{BTV95}).
The classification of generalized Heisenberg algebras is known (it
follows from the classification of representations of Clifford
algebras of vector spaces with negative definite quadratic forms). In
particular, for each $m\in\mathbb{N}$ there exist an infinite number
of non-isomorphic generalized Heisenberg algebras with $\dim
\g{z}=m$.

\subsection{Damek-Ricci spaces}\label{secDamekRicci}

Let $\g{a}$ be a one-dimensional real vector space, $B$ a non-zero
vector in $\g{a}$ and $\g{n}=\g{v}\oplus\g{z}$ a generalized
Heisenberg algebra, where $\g{z}$ is the center of $\g{n}$. We denote
the inner product and the Lie bracket on $\g{n}$ by
$\langle\cdot,\cdot\rangle_{\g{n}}$ and $[\cdot,\cdot]_{\g{n}}$,
respectively, and consider a new vector space $\g{a}\oplus\g{n}$ as
the vector space direct sum of $\g{a}$ and $\g{n}$.

From now on in this section, let $s$, $r\in\R$, $U$, $V\in\g{v}$ and
$X$, $Y \in\g{z}$. We now define an inner product
$\langle\cdot,\cdot\rangle$ and a Lie bracket $[\cdot,\cdot]$ on
$\g{a}\oplus\g{n}$ by
\begin{align*}
&\langle rB+U+X,sB+V+Y\rangle=rs+\langle U+X,V+Y\rangle_\g{n},\quad\text{ and }\\
&[rB+U+X,sB+V+Y]=[U,V]_\g{n}+\frac{1}{2}rV-\frac{1}{2}sU+rY-sX.
\end{align*}
Thus, $\g{a}\oplus\g{n}$ becomes a solvable Lie algebra with an
inner product. The corresponding simply connected Lie group $AN$,
equipped with the induced left-invariant Riemannian metric, is a
solvable extension of the H-type group $N$, and is called a
\emph{Damek-Ricci space}.

The Levi-Civita connection $\nabla$ of a Damek-Ricci space is given by
\[
\nabla_{sB+V+Y}(rB+U+X)=-\frac{1}{2}J_X V-\frac{1}{2}J_Y U-\frac{1}{2}rV
-\frac{1}{2}[U,V]-rY+\frac{1}{2}\langle U,V\rangle B+\langle X,Y\rangle B.
\]

From this expression, one can obtain the curvature tensor $R$ of
$AN$, where we agree to take the convention
$R(W_1,W_2)=[\nabla_{W_1}, \nabla_{W_2}] -\nabla_{[W_1,W_2]}$.

A Damek-Ricci space $AN$ is a symmetric space if and only if $AN$ is
isometric to a rank one symmetric space. In this case, $AN$ is either
isometric to a complex hyperbolic space $\C H^n$ with constant
holomorphic sectional curvature $-1$ (in this case, $\dim \g{z}=1$),
or to a quaternionic hyperbolic space $\mathbb{H}H^n$ with constant
quaternionic sectional curvature $-1$ (here $\dim \g{z}=3$), or to
the Cayley hyperbolic plane $\mathbb{O}H^2$ with minimal sectional
curvature $-1$ ($\dim \g{z}=7$). As a limit case, which we will
disregard in what follows, one would obtain the real hyperbolic space
$\R H^n$ if one puts $\g{z}=0$.

The non-symmetric Damek-Ricci spaces are counterexemples to the
so-called Lichnerowicz  conjecture, stating that every Riemannian
harmonic manifold is locally isometric to a two-point homogeneous
space. There are several equivalent conditions for a manifold to be
harmonic; see \cite[\S 2.6]{BTV95}. One of them is the following: a
manifold is harmonic if and only if its sufficiently small geodesic
spheres are isoparametric. However, while geodesic spheres in
symmetric Damek-Ricci spaces are homogeneous isoparametric
hypersurfaces with constant principal curvatures, geodesic spheres in
non-symmetric Damek-Ricci spaces are inhomogeneous isoparametric
hypersurfaces with nonconstant principal curvatures (see \cite{DR92}
and \cite[\S 4.4 and \S 4.5]{BTV95}).

\section{Generalized K\"{a}hler angle}\label{secGKA}

In this section we introduce the new notion of generalized K\"{a}hler
angle of a vector of a subspace of a Clifford module with respect to
that subspace. This notion will be crucial for the rest of the work.

Let $\g{v}$ be a Clifford module over $Cl(\g{z},q)$ and denote by
$J\colon\g{z}\to\textup{End}(\g{v})$ the restriction to $\g{z}$ of
the Clifford algebra representation. We equip $\g{z}$ with the inner
product induced by polarization of $-q$, and extend it to an inner
product $\langle\,\cdot\,,\,\cdot\,\rangle$ on
$\g{n}=\g{v}\oplus\g{z}$, so that $\g{v}$ and $\g{z}$ are
perpendicular, and $J_Z$ is an orthogonal map for each unit
$Z\in\g{z}$. Then, $\g{n}$ has the structure of a generalized
Heisenberg algebra as defined above.

Let $\g{w}$ be a subspace of $\g{v}$. We denote by
$\g{w}^\perp=\g{v}\ominus\g{w}$ the orthogonal complement of $\g{w}$
in $\g{v}$. For each $Z\in\g{z}$ and $\xi\in\g{w}^\perp$, we write
$J_Z \xi= P_Z \xi+F_Z \xi$, where $P_Z\xi$ is the orthogonal
projection of $J_Z\xi$ onto $\g{w}$, and $F_Z\xi$ is the orthogonal
projection of $J_Z\xi$ onto $\g{w}^\perp$. We define the K\"{a}hler angle
of $\xi\in\g{w}^\perp$ with respect to the element $Z\in\g{z}$ (or,
equivalently, with respect to $J_Z$) and the subspace
$\g{w}^\perp\subset \g{v}$ as the angle $\varphi\in[0,\pi/2]$ between
$J_Z\xi$ and $\g{w}^\perp$~\cite{BB01}; thus $\varphi$ satisfies
$\langle F_Z\xi,F_Z\xi\rangle =\cos^2(\varphi)\langle
Z,Z\rangle\langle \xi,\xi\rangle$. It readily follows from
$J_Z^2=-\langle Z,Z\rangle \id_\g{v}$ that $\langle
P_Z\xi,P_Z\xi\rangle =\sin^2(\varphi)\langle Z,Z\rangle \langle
\xi,\xi\rangle$. Hence, if $Z$ and $\xi$ have unit length, $\varphi$
is determined by the fact that $\cos(\varphi)$ is the length of the
orthogonal projection of $J_Z\xi$ onto $\g{w}^\perp$.

The following theorem is a generalization of \cite[Lemma 3]{BB01}
(which concerned only the case of the quaternionic hyperbolic space
$\mathbb{H}H^n$). The proof is new and simpler than in \cite{BB01}.
This result will be fundamental for the calculations we will carry
out later.

\begin{theorem} \label{thGKA}
Let $\g{w}^\perp$ be some vector subspace of $\g{v}$ and let
$\xi\in\g{w}^\perp$ be a nonzero vector. Then there exists an
orthonormal basis $\{Z_1,\dots, Z_m\}$ of $\g{z}$ and a uniquely
defined $m$-tuple $(\varphi_1,\dots,\varphi_m)$ such that:
\begin{enumerate}[{\rm (a)}]
\item $\varphi_i$ is the K\"ahler angle of $\xi$ with respect to
    $J_{Z_i}$, for each $i=1,\dots, m$.\label{thGKA:1}
\item $\langle P_{Z_i}\xi,P_{Z_j}\xi\rangle=\langle
    F_{Z_i}\xi,F_{Z_j}\xi\rangle=0$ whenever $i\neq j$.
\item
    $0\leq\varphi_1\leq\varphi_2\leq\dots\leq\varphi_m\leq\pi/2$.
\item $\varphi_1$ is minimal and $\varphi_m$ is maximal among the
    K\"ahler angles of $\xi$ with respect to all the elements of
    $\g{z}$.\label{thGKA:4}
\end{enumerate}
\end{theorem}

\begin{proof}
Since the map $Z\in\g{z}\mapsto F_Z\xi\in\g{w}^\perp$ is linear, we
can define the quadratic form
\[
Q_\xi\colon Z\in \g{z}\mapsto \langle F_Z\xi,F_Z\xi\rangle\in\R.
\]
Observe that $\varphi$ is the K\"ahler angle of $\xi$ with respect to
$Z\in \g{z}$ ($Z\neq 0$) and the subspace $\g{w}^\perp\subset\g{v}$
if and only if $Q_\xi(Z)=\cos^2(\varphi)\langle Z, Z\rangle \langle
\xi, \xi \rangle$.

Let $\{Z_1,\dots, Z_m\}$ be an orthonormal basis of $\g{z}$ for which
the quadratic form $Q_\xi$ assumes a diagonal form. Define the real
numbers $\varphi_1,\dots,\varphi_m\in[0,\pi/2]$ by the expression
$Q_\xi(Z_i)=\cos^2(\varphi_i)\langle \xi, \xi \rangle$, for every
$i=1,\dots, m$. We can further assume that $\varphi_1\leq
\dots\leq\varphi_m$, by reordering the elements of the basis in a
suitable way.

If $L$ is the symmetric bilinear form associated with $Q_\xi$, then
$L_\xi(X,Y)=\langle F_X\xi, F_Y\xi\rangle$, for each $X$,
$Y\in\g{z}$. But then the fact that $\{Z_1,\dots, Z_m\}$ is an
orthonormal basis for which $Q_\xi$ assumes a diagonal form is
equivalent to $0=L_\xi(Z_i,Z_j)=\langle F_{Z_i}\xi,
F_{Z_j}\xi\rangle$ for all $i\neq j$. This, together with the
ordering of $(\varphi_1,\dots,\varphi_m)$ and the fact that
$\{Z_1,\dots, Z_m\}$ is an orthonormal basis, implies that the
$m$-tuple $(\varphi_1,\dots,\varphi_m)$ is uniquely defined for a
fixed $\g{w}^\perp$ and a fixed $\xi\in\g{w}^\perp$. Moreover, due to
the bilinearity of $L_\xi$, it is clear that $\varphi_1$ is minimal
and $\varphi_m$ is maximal among the K\"ahler angles of $\xi$ with
respect to all the elements of $\g{z}$. Finally, we also have that
$\langle P_{Z_i}\xi,P_{Z_j}\xi\rangle=\langle
J_{Z_i}\xi,J_{Z_j}\xi\rangle-\langle F_{Z_i}\xi,F_{Z_j}\xi\rangle=0$,
whenever $i\neq j$.
\end{proof}

Motivated by Theorem~\ref{thGKA}, we define the \emph{generalized
K\"ahler angle} of $\xi$ with respect to $\g{w}^\perp$ as the
$m$-tuple $(\varphi_1,\dots,\varphi_m)$ satisfying properties
(\ref{thGKA:1})-(\ref{thGKA:4}) of Theorem~\ref{thGKA}.

\begin{remark}
Observe that the K\"ahler angles $\varphi_1,\dots,\varphi_m$ depend,
not only on the subspace $\g{w}^\perp$ of $\g{v}$, but also on the
vector $\xi\in\g{w}^\perp$.
\end{remark}

Assuming the notation of the previous theorem, we will say that the
subspace $\g{w}^\perp$ of $\g{v}$ has \emph{constant generalized
K\"ahler angle} $(\varphi_1,\dots,\varphi_m)$ if the $m$-tuple
$(\varphi_1,\dots,\varphi_m)$ is independent of the unit vector
$\xi\in\g{w}^\perp$.

If $\g{v}=\C^n$ and $\g{z}=\R$, then the complex structure of $\C^n$
is $J=J_1$. For a given subspace $\g{w}$ of $\C^n$, we denote $F=F_1$
and $P=P_1$, and we define $\bar{F}\xi=F\xi/\lVert F\xi\rVert$ if
$F\xi\neq 0$. We will need the following result from~\cite[Lemma
2]{BB01}:

\begin{lemma}\label{lemma:BB}
Let $\g{w}^\perp$ be some linear subspace of $\C^n$, and
$\xi\in\g{w}^\perp$ a unit vector with K\"{a}hler angle $\varphi\in
(0,\pi/2)$. Then, there exists a unique vector $\eta\in\C^n\ominus
\C\xi$ such that $\bar{F}\xi=\cos(\varphi)J\xi+\sin(\varphi)J\eta$.
\end{lemma}

\section{The new examples}\label{sec:Examples}

The new isoparametric hypersurfaces will be tubes around certain
homogeneous submanifolds of a Damek-Ricci space. Thus, in this
section, we proceed first with the construction of these submanifolds
and then determine their extrinsic geometry. This is done in
Subsection~\ref{secFocal}. The geometry of the tubes around these
focal submanifolds is studied in Subsection~\ref{secJacobi}, where
their main properties are given.

\subsection{The focal manifold of the new examples}\label{secFocal}

As we explained above, the new examples are constructed as tubes
around certain homogeneous submanifolds. Each isoparametric family
will have at most one submanifold that is not a hypersurface. This is
the focal submanifold of the family, and we define it in this
subsection.

Let $AN$ be a Damek-Ricci space with Lie algebra
$\g{a}\oplus\g{n}=\g{a}\oplus\g{v}\oplus\g{z}$, where $\dim \g{z}=m$.
Let $\g{w}$ be a proper subspace of $\g{v}$ and define
$\g{w}^\perp=\g{v}\ominus\g{w}$, the orthogonal complement of $\g{w}$
in $\g{v}$. Then,
\[
\g{s}_{\g{w}}=\g{a}\oplus\g{w}\oplus\g{z}
\]
is a solvable Lie subalgebra of $\g{a}\oplus\g{n}$, as one can easily
check from the bracket relations in \S\ref{secDamekRicci}. Let
$S_{\g{w}}$ be the corresponding connected subgroup of $AN$ whose Lie
algebra is $\g{s}_{\g{w}}$. Since $AN$ acts by isometries on itself
and $S_{\g{w}}$ is a subgroup of $AN$, $S_\g{w}$ is also a
homogeneous submanifold of $AN$.

Let $\xi\in\g{w}^\perp$ be a unit normal vector field along the
submanifold $S_\g{w}$. Let $\{Z_1,\dots,Z_m\}$ be an orthonormal
basis of $\g{z}$ satisfying the properties of Theorem~\ref{thGKA}. In
order to simplify the notation, for each $i\in\{1,\dots,m\}$, we set
$J_i$, $P_i$ and $F_i$ instead of $J_{Z_i}$, $P_{Z_i}$ and $F_{Z_i}$,
respectively. It is convenient to define
\[
\mi=\max\{i:\varphi_i=0\}+1\quad \text{ and }\quad
\mii=\min\{i:\varphi_i=\pi/2\}-1,\] where $\varphi_i$ is the K\"ahler
angle of $\xi$ with respect to $Z_i\in \g{z}$ (set $\mi=1$ if
$\varphi_i>0$ for all $i$, and $\mii=m$ if $\varphi_i<\pi/2$ for all
$i$). Thus, $\mi$ is the first index $i$ for which $\varphi_i>0$, and
$\mii$ is the last index $i$ for which $\varphi_i<\pi/2$. It might of
course happen that $\mi>m$ if $\varphi_i=0$ for all $i$, or $\mii<1$
if $\varphi_i=\pi/2$ for all $i$, in which case some of the equations
that follow are just disregarded.

With this notation we can now define
\[
\bar{P}_i\xi=\frac{1}{\sin(\varphi_i)}P_i\xi, \text{ for }i=\mi,\dots, m,
\quad\text{ and }\quad
\bar{F}_i\xi=\frac{1}{\cos(\varphi_i)}F_i\xi, \text{ for } i=1,\dots, \mii.
\]
Since $\xi$ is of unit length, so are $\bar{P}_i\xi$ and
$\bar{F}_i\xi$ whenever they exist. Moreover, by Theorem~\ref{thGKA},
the set $\{\bar{P}_{\mi}\xi,\dots, \bar{P}_m\xi, \bar{F}_1\xi,\dots,
\bar{F}_{\mii}\xi\}$ constitutes an orthonormal system of vector
fields along $S_\g{w}$, the first $m-\mi+1$ of which being tangent,
and the rest normal to~$S_\g{w}$.

We are now interested in calculating the shape operator $\mathcal{S}$
of $S_\g{w}$. Recall that the shape operator $\mathcal{S}_\xi$ of
$S_{\g{w}}$ with respect to a unit normal $\xi\in\nu S_{\g{w}}$ is
defined by $\mathcal{S}_\xi X=-(\nabla_X\xi)^\top$, for any $X\in
TS_{\g{w}}$, and where $(\cdot)^\top$ denotes orthogonal projection
onto the tangent space. The expression for the Levi-Civita connection
of the Damek-Ricci space $AN$ allows us to calculate the shape
operator of $S_\g{w}$ for left-invariant vector fields:
\begin{align*}
\mathcal{S}_\xi B & =0,
\\
\mathcal{S}_\xi Z_i &
=\frac{1}{2}P_i\xi=0,\quad\text{ if $i=1,\dots, \mi-1$},
\\
\mathcal{S}_\xi Z_i &
=\frac{1}{2}P_i\xi=\frac{1}{2}\sin(\varphi_i)\bar{P}_i\xi,\quad\text{ if $i=\mi,\dots, m$},
\\
\mathcal{S}_\xi \bar{P}_i\xi &
= \frac{1}{2}[\xi, \bar{P}_i\xi]^\top
= \frac{1}{2}\sum_{j=1}^m \langle J_j\xi, \bar{P}_i\xi\rangle Z_j
=\frac{1}{2}\sin(\varphi_i)Z_i, \quad \text{ if $i=\mi,\dots, m$},
\\
\mathcal{S}_\xi U & = \frac{1}{2}[\xi, U]^\top
= \frac{1}{2}\sum_{j=1}^m \langle J_j\xi, U\rangle Z_j=0 ,
\quad \text{ if $U\in\g{w}\ominus\Bigl(\bigoplus_{i=\mi}^m \R\bar{P}_i\xi\Bigr)$}.
\end{align*}

From the expressions above, we obtain that the principal curvatures
of $S_\g{w}$ with respect to the unit normal vector $\xi$ are
\[
0, \quad \frac{1}{2}\sin\varphi_i,\quad \text{ and }\quad-\frac{1}{2}\sin\varphi_i,
\]
and their corresponding principal spaces are, respectively,
\begin{align*}
\g{a}\oplus\Bigl(\g{w}\ominus\Bigl(\bigoplus_{j=\mi}^m \R\bar{P}_j\xi\Bigr)\Bigr)
\oplus\Bigl(\bigoplus_{j=1}^{\mi-1}Z_j\Bigr),
\quad  \R(Z_i+\bar{P}_i\xi), \quad\text{ and }\quad \R(Z_i-\bar{P}_i\xi),
\end{align*}
where $i=\mi,\dots, m$. In any case, the submanifold $S_\g{w}$ is
minimal (even austere) and, if $\dim\g{w}^\perp=1$, then $S_\g{w}$ is
a minimal hypersurface of $AN$.

\begin{remark}
We emphasize that, although the dependance on $\xi$ is not made
explicit in the notation, $(\varphi_1,\dots,\varphi_m)$,
$\{Z_1,\dots,Z_m\}$, $\mi$, and $\mii$ do depend on $\xi$.
\end{remark}

\subsection{Solving the Jacobi equation}\label{secJacobi}

Denote by $M^r$ the tube of radius $r$ around the submanifold
$S_\g{w}$ that was described in the previous section. We claim that,
for every $r>0$, $M^r$ is an isoparametric hypersurface which has, in
general, nonconstant principal curvatures.

In order to show that $M^r$ has the properties mentioned above, we
will make use of Jacobi field theory. The main step of our approach
is to write down the Jacobi equation along a geodesic normal to
$S_\g{w}$ and to solve some initial value problems for this equation.
We emphasize that the method used in~\cite{DRDV10} to write down the
Jacobi equation corresponding to complex hyperbolic space $\C H^n$
(which consisted in expressing the Jacobi fields in terms of parallel
translations) is not feasible here. Thus we will express the Jacobi
fields in terms of left-invariant vector fields. The relevance of
Theorem~\ref{thGKA} will become clear with this method.

Given a unit speed geodesic $\gamma$ in the Damek-Ricci space $AN$, a
vector field $\zeta$ along $\gamma$ is called a Jacobi vector field
if it satisfies the Jacobi equation in $AN$ along $\gamma$, namely
\[
\zeta''+R(\zeta,\dot{\gamma})\dot{\gamma}=0,
\]
where $\dot\gamma$ is the tangent vector of $\gamma$, and $'$ stands
for covariant differentiation along the geodesic $\gamma$.

Let $p\in S_\g{w}$ be an element of the submanifold, and $\xi\in\nu_p
S_\g{w}$ a unit normal vector at $p$. Let $\gamma$ be the geodesic of
$AN$ such that $\gamma(0)=p$ and $\dot{\gamma}(0)=\xi$. Denote by
$\dot{\gamma}(t)^\perp$ the orthogonal complement of
$\dot{\gamma}(t)$ in $T_{\gamma(t)}AN$ and by $\mathcal{S}$ the shape
operator of the submanifold $S_\g{w}$.  We denote by $\zeta_v$ the
Jacobi vector field with initial conditions
\begin{equation}\label{initialConditions}
\zeta_v(0)=v^\top,\, \zeta_v'(0)=-\mathcal{S}_\xi v^\top +v^\perp,
\text{ where $v=v^\top+v^\perp$, $v^\top\in\g{s}_\g{w}$, and
$v^\perp\in \g{w}^\perp\ominus\R\xi$}.
\end{equation}
We define the $\textrm{Hom}\bigl((\g{a}\oplus\g{n})\ominus\R
\xi,\dot{\gamma}^\perp\bigr)$-valued tensor fields $C$ and $E$ along
$\gamma$ satisfying $C(r) v=\zeta_v(r)$ and $E(r)
v=(\zeta_v'(r))^\top$, for every $r\in\R$ and every left-invariant
vector field $v\in(\g{a}\oplus\g{n})\ominus\R \xi$, where now
$(\cdot)^\top$ denotes the projection onto $\dot{\gamma}^\perp$.
Standard Jacobi field theory ensures that if $C(r)$ is nonsingular
for every unit $\xi\in\nu S_\g{w}$, then the tube $M^r$ of radius $r$
around $S_\g{w}$ is a hypersurface of $AN$. Moreover, in this case,
the shape operator $\mathcal{S}^r$ of $M^r$ at the point $\gamma(r)$
with respect to the unit vector $-\dot{\gamma}(r)$ is given by
$\mathcal{S}^r\zeta_v(r)=(\zeta_v'(r))^\top$, for every
$v\in(\g{a}\oplus\g{n})\ominus\R \xi$, that is, $\mathcal{S}^r =
E(r)C(r)^{-1}$.

Therefore, our objective in what follows is to determine an explicit
expression for the Jacobi fields whose initial conditions are given
by~\eqref{initialConditions}. In order to achieve this goal, we fix
here and henceforth an orthonormal basis $\{Z_1,\dots,Z_m\}$ of
$\g{z}$ satisfying the properties of Theorem~\ref{thGKA}, and let
$(\varphi_1,\dots,\varphi_m)$ be the corresponding generalized
K\"ahler angle of $\xi$ with respect to $\g{w}^\perp$. Recall that
$(\varphi_1,\dots,\varphi_m)$ and $\{Z_1,\dots,Z_m\}$ depend on
$\xi$, but we remove this dependence from the notation for the sake
of simplicity. Let $\{U_1,\dots, U_l\}$ be an orthonormal basis of
$\g{w}\ominus\bigl(\oplus_{j=\mi}^m \R\bar{P}_j\xi\bigr)$, and let
$\{\eta_1,\dots,\eta_h\}$ be an orthonormal basis of
$\g{w}^\perp\ominus\bigl(\R\xi\oplus\bigl(\oplus_{j=1}^{\mii}\bar{F}_j\xi\bigr)\bigr)$.
Then the set
\begin{equation}\label{basis}
\Bigl\{\,\xi, B, U_1,\dots, U_l,\eta_1,\dots, \eta_h, \bar{P}_{\mi}\xi,
\dots, \bar{P}_m\xi,\bar{F}_1\xi,\dots, \bar{F}_{\mii}\xi, Z_1, \dots, Z_m\,\Bigr\}
\end{equation}
constitutes an orthonormal basis of left-invariant vector fields of
$\g{a}\oplus\g{n}$.

The main step of the proof of Theorem~\ref{thMain} is the following

\begin{proposition}\label{prop:JacobiSolution}
With the notation as above we have
\begin{align*}
\zeta_{B}(t)={}& B+ \sinh\Bigl(\frac{t}{2}\Bigr)\xi ,
\\
\zeta_{U_i}(t)={}& \cosh\Bigl(\frac{t}{2}\Bigr) U_i, \quad i=1,\dots, l,
\\
\zeta_{\eta_i}(t)={}& 2\sinh\Bigl(\frac{t}{2}\Bigr) \eta_i, \quad i=1,\dots, h,
\\
\zeta_{\bar{P}_i\xi}(t)={}& \cosh\Bigl(\frac{t}{2}\Bigr) \bar{P}_i\xi
- \sin(\varphi_i)\sinh(t) Z_i, \quad i=\mi,\dots, m,
\\
\zeta_{\bar{F}_i\xi}(t)={}& 2\sinh\Bigl(\frac{t}{2}\Bigr) \bar{F}_i\xi
-2\cos(\varphi_i)\sinh^2\Bigl(\frac{t}{2}\Bigr)Z_i, \quad i=1,\dots, \mii,
\\
\zeta_{Z_i}(t)={}& \sinh\Bigl(\frac{t}{2}\Bigr) {F}_i\xi
+\Bigl(1+\sin^2(\varphi_i)\sinh^2\Bigl(\frac{t}{2}\Bigr)\Bigr)Z_i, \quad i=1,\dots, m.
\end{align*}
\end{proposition}

\begin{proof}
In order to prove this result it suffices to take the expressions
above and show that they satisfy the Jacobi equation and the initial
conditions~\eqref{initialConditions}. The calculations are long so we
will just show an example of how they are performed for
$\zeta_{Z_i}$.

First of all, recall that $p\in S_\g{w}$, and $\xi\in\nu_p S_\g{w}$
is a unit normal vector at $p$. The geodesic $\gamma$ satisfies
$\gamma(0)=p$ and $\dot{\gamma}(0)=\xi$. By \cite[\S 4.1.11, Theorem
2]{BTV95} we know that $\dot{\gamma}(t)=\sech\bigl(\frac{t}{2}\bigr)
\xi -\tanh\bigl(\frac{t}{2}\bigr) B$, for every $t\in\R$, where $\xi$
and $B$ are considered as left-invariant vector fields on $AN$.
Actually, in \cite{BTV95} this result is stated only for the case
when $p=e$ is the identity element of $AN$. However, since
$\gamma_p=L_p\circ\gamma_e$, the homogeneity of $S_\g{w}$ implies
that
\begin{equation}\label{normalGeodesic}
\dot{\gamma}_p(t)=L_{p*}\dot{\gamma}_e(t)
=\sech\Bigl(\frac{t}{2}\Bigr) \xi -\tanh\Bigl(\frac{t}{2}\Bigr) B,
\end{equation}
for every $t\in\R$, where $L_p$ denotes the left multiplication by
$p$ in the group $AN$, $\gamma_e$ is the normal geodesic through the
identity element $e$ with initial velocity $\xi$, and $\gamma_p$ is
the normal geodesic through the point $p\in AN$ with initial velocity
$L_{p*}\xi=\xi$.

It is easy to check that $\zeta_{Z_i}(0)=Z_i$, which is a tangent
vector to $\g{s}_{\g{w}}$. Now we have to calculate $\zeta_{Z_i}'$.
By the Leibniz rule we get
\begin{equation}\label{eq:zetap}
\begin{aligned}
\zeta_{Z_i}'(t)
=\nabla_{\dot\gamma(t)}\zeta_{Z_i}
={}&\frac{1}{2}\cosh\Bigl(\frac{t}{2}\Bigr)F_i\xi
+\sinh\Bigl(\frac{t}{2}\Bigr)\nabla_{\dot\gamma(t)}F_i\xi
+\sin^2(\varphi_i)\sinh\Bigl(\frac{t}{2}\Bigr)\cosh\Bigl(\frac{t}{2}\Bigr)Z_i\\
&{}+\Bigl(1+\sin^2(\varphi_i)\sinh^2\Bigl(\frac{t}{2}\Bigr)\Bigr)
\nabla_{\dot\gamma(t)}Z_i.
\end{aligned}
\end{equation}
Using~\eqref{normalGeodesic}, and the formula for the Levi-Civita
connection in~\S\ref{secDamekRicci}, we obtain
\begin{align*}
\nabla_{\dot\gamma(t)}F_i\xi
&{}=\sech\Bigl(\frac{t}{2}\Bigr)\nabla_{\xi}F_i\xi
-\tanh\Bigl(\frac{t}{2}\Bigr)\nabla_{B}F_i\xi
=\frac{1}{2}\sech\Bigl(\frac{t}{2}\Bigr)[\xi,F_i\xi],\quad
\text{ and}\\
\nabla_{\dot\gamma(t)}Z_i
&{}=\sech\Bigl(\frac{t}{2}\Bigr)\nabla_{\xi}Z_i
-\tanh\Bigl(\frac{t}{2}\Bigr)\nabla_{B}Z_i
=-\frac{1}{2}\sech\Bigl(\frac{t}{2}\Bigr)J_{Z_i}\xi.
\end{align*}
Now, for $j\in\{1,\dots,m\}$, the bracket relations
from~\S\ref{sec:Heisenberg} yield $\langle[\xi,F_i\xi],Z_j\rangle
=\langle J_{Z_j}\xi,F_i\xi\rangle =\langle F_j\xi,F_i\xi\rangle
=\cos^2(\varphi_i)\delta_{ij}$, where $\delta$ is the Kronecker
delta. Thus, $[\xi,F_i\xi]=\cos^2(\varphi_i)Z_i$. Using
$J_{Z_i}\xi=P_i\xi+F_i\xi$, and inserting these expressions
into~\eqref{eq:zetap} we obtain
\[
\begin{aligned}
\zeta_{Z_i}'(t)=
&{}-\frac{1}{2}\sech\Bigl(\frac{t}{2}\Bigr)
\Bigl(1+\sin^2(\varphi_i)\sinh^2\Bigl(\frac{t}{2}\Bigr)\Bigr){P}_i\xi
+\frac{1}{2}\cos^2(\varphi_i)\sinh\Bigl(\frac{t}{2}\Bigr)
\tanh\Bigl(\frac{t}{2}\Bigr){F}_i\xi\\
{}&{}+\frac{1}{2}\Bigl(\sin^2(\varphi_i)\sinh(t)+\cos^2(\varphi_i)
\tanh\Bigl(\frac{t}{2}\Bigr)\Bigr)Z_i, \quad i=1,\dots, m.
\end{aligned}
\]
From this expression, and the shape operator of $S_{\g{w}}$ obtained
in \S\ref{secFocal} we easily get
$\zeta_{Z_i}'(0)=-\frac{1}{2}P_i\xi=-\frac{1}{2}\sin\varphi_i\bar{P}_i\xi
=-\mathcal{S}_\xi Z_i$, so the initial
conditions~\eqref{initialConditions} are satisfied.

The very same approach can be used to calculate $\zeta_{Z_i}''$. We
omit the explicit calculations here, which are very similar to those
shown above, and give the result:
\begin{align*}
\zeta_{Z_i}''(t)={}
&{}-\frac{3}{4}\sin^2(\varphi_i)\sinh\Bigl(\frac{t}{2}\Bigr)P_i\xi
+\frac{1}{16}(6\cos(2\varphi_i)-2)\sinh\Bigl(\frac{t}{2}\Bigr)F_i\xi\\
&{}+\frac{1}{4}\Bigl(\cosh(t)-2\cos(2\varphi_i)\sinh^2\Bigl(\frac{t}{2}\Bigr)\Bigr)Z_i.
\end{align*}

Finally, we need to calculate
$R(\zeta_{Z_i}(t),\dot\gamma(t))\dot\gamma(t))$. We have the
following identities for the curvature tensor, where $U$,
$V\in\g{v}$, and $Z\in\g{z}$ are of unit length (for the complete
formula, see \cite[\S 4.1.7]{BTV95}):
\begin{align*}
R(U,Z)B &= \frac{1}{4}J_Z U,& R(B,Z)B &= Z,
&R(U,V)V &= \frac{1}{4}\bigl(\langle U,V\rangle V- U+3 J_{[U,V]}V\bigr),\\
R(U,V)B &=\frac{1}{2}[U,V],\quad&
R(B,U)B &= \frac{1}{4}U,\quad&
R(U,B)V &= \frac{1}{4}\langle U,V\rangle B +\frac{1}{4}[U,V],\\
R(B,Z)U &= \frac{1}{2}J_Z U,&
R(U,Z)U &= \frac{1}{4} Z,& R(U,B)U &= \frac{1}{4} B.
\end{align*}

Using the properties of the curvature tensor and the formulas above,
we get after some calculations
$\zeta_{Z_i}''+R(\zeta_{Z_i},\dot\gamma)\dot\gamma=0$ as we wanted to
show.
\end{proof}

Our aim in what follows is to finish the calculation of the shape
operator $\mathcal{S}^r$ of $M^r$ at $\gamma(r)$. Recall that we
first need to calculate
$C(r)\colon(\g{a}\oplus\g{n})\ominus\R\xi\to\dot\gamma^\perp=T_{\gamma(r)}M^r$,
$v\mapsto\zeta_v(r)$. In order to describe this operator we consider
the following distributions on $\g{a}\oplus\g{n}$:
\begin{align*}
\g{U}&=\oplus_{j=1}^l\R U_j,\qquad
&\g{F}_i&=\R \bar{F}_i\xi\oplus\R Z_i,&& i=1,\dots,\mi-1,\\
\g{H}&=\oplus_{j=1}^h\R \eta_j,
&\g{M}_i&=\R \bar{P}_i\xi\oplus\R\bar{F}_i\xi\oplus\R Z_i,&& i=\mi,\dots,\mii,\\
&&\g{P}_i&=\R \bar{P}_i\xi\oplus\R Z_i,&& i=\mii+1,\dots,m.
\end{align*}
Then, we can decompose
\begin{align*}
&(\g{a}\oplus\g{n})\ominus\R\xi
=\R B\oplus\g{U}\oplus\g{H}
\oplus\Bigl(\bigoplus_{i=1}^{\mi-1}\g{F}_i\Bigr)
\oplus\Bigl(\bigoplus_{i=\mi}^{\mii}\g{M}_i\Bigr)
\oplus\Bigl(\bigoplus_{i=\mii+1}^{m}\g{P}_i\Bigr),\text{ and}\\
&T_{\gamma(r)}M^r
=\R\Bigl(\sech\Bigl(\frac{r}{2}\Bigr)B+\tanh\Bigl(\frac{r}{2}\Bigr)\xi\Bigr)
\oplus\g{U}\oplus\g{H}
\!\oplus\!\Bigl(\bigoplus_{i=1}^{\mi-1}\g{F}_i\Bigr)
\!\oplus\!\Bigl(\bigoplus_{i=\mi}^{\mii}\g{M}_i\Bigr)
\!\oplus\!\Bigl(\bigoplus_{i=\mii+1}^{m}\g{P}_i\Bigr).
\end{align*}
With respect to these decompositions, a direct application of
Proposition~\ref{prop:JacobiSolution} shows that the operator $C(r)$
can be written as
\begin{align*}
C(r)={}&\cosh\Bigl(\frac{r}{2}\Bigr)\id_{l+1}
\oplus\Bigl(2\sinh\Bigl(\frac{r}{2}\Bigr)\id_{h}\Bigr)
\oplus\left(\bigoplus_{i=1}^{\mi-1}
\begin{pmatrix}
2\sinh(\frac{r}{2}) &\sinh(\frac{r}{2})\\
-2\sinh^2(\frac{r}{2}) & 1
\end{pmatrix}
\right)\\
&{}\oplus\left(\bigoplus_{i=\mi}^{\mii}
\begin{pmatrix}
\cosh(\frac{r}{2}) & 0 & 0\\
 0 &  2\sinh(\frac{r}{2}) & \cos(\varphi_i)\sinh(\frac{r}{2})\\
-\sin(\varphi_i)\sinh(r) & -2\cos(\varphi_i)\sinh^2(\frac{r}{2})
& 1+\sin^2(\varphi_i)\sinh^2(\frac{r}{2})
\end{pmatrix}
\right)\\
&{}\oplus\left(\bigoplus_{i=\mii+1}^{m}
\begin{pmatrix}
\cosh(\frac{r}{2}) & 0\\
-\sinh(r) & \cosh^2(\frac{r}{2})
\end{pmatrix}
\right).
\end{align*}

In particular, the determinant of $C(r)$ is
\[
\det(C(r))=2^{h+\mii}\Bigl(\cosh\Bigl(\frac{r}{2}\Bigr)\Bigr)^{2+l+3m-\mi}
\Bigl(\sinh\Bigl(\frac{r}{2}\Bigr)\Bigr)^{h+\mii},
\]
which is nonzero for every $r> 0$, and hence, the tubes $M^r$ around
$S_\g{w}$ are hypersurfaces for every $r>0$.

The next step is to consider the operator
$E(r)\colon(\g{a}\oplus\g{n})\ominus\R\xi\to T_{\gamma(r)}M^r$,
$v\mapsto(\zeta_v'(r))^\top$. To that end, we need to calculate the
covariant derivative along $\gamma$ of the Jacobi vector fields given
in Proposition~\ref{prop:JacobiSolution}. We omit the explicit
calculations, which follow the procedure described in the proof of
Proposition~\ref{prop:JacobiSolution}, and give $E(r)$ with respect
to the same decomposition as above:
\begin{align*}
E(r)={}&\frac{1}{2}\sinh\Bigl(\frac{r}{2}\Bigr)\id_{l+1}
\oplus\Bigl(\cosh\Bigl(\frac{r}{2}\Bigr)\id_{h}\Bigr)
\oplus\left(\bigoplus_{i=1}^{\mi-1}
\begin{pmatrix}
\cosh(r)\sech(\frac{r}{2}) &\frac{1}{2}\sinh(\frac{r}{2})\tanh(\frac{r}{2})\\
\tanh(\frac{r}{2})-\sinh(r) & \frac{1}{2}\tanh(\frac{r}{2})
\end{pmatrix}
\right)\\
&{}\oplus\left(\bigoplus_{i=\mi}^{\mii}
A_i(r)\right)
\oplus\left(\bigoplus_{i=\mii+1}^{m}
\begin{pmatrix}
\frac{3}{2}\sinh(\frac{r}{2}) & -\frac{1}{2}\cosh(\frac{r}{2})\\[0.5ex]
\frac{1}{2}-\cosh(r) & \frac{1}{2}\sinh(r)
\end{pmatrix}
\right),
\end{align*}
where $A_i(r)$ is
{\Small
\[
\frac{1}{2}\begin{pmatrix}
(2-\cos(2\varphi_i))\sinh(\frac{r}{2}) &
2\sin(2\varphi_i)\csch(r)\sinh^3(\frac{r}{2})
&  -\sin(\varphi_i)\sech(\frac{r}{2})(1+\sin^2(\varphi_i)\sinh^2(\frac{r}{2}))\\[0.5ex]
\sin(2\varphi_i)\sinh(\frac{r}{2})
& 2\cosh(\frac{r}{2})(1+\cos^2(\varphi_i)\tanh^2(\frac{r}{2}))
& \cos^3(\varphi_i)\sinh(\frac{r}{2})\tanh(\frac{r}{2})\\[0.5ex]
\sin(\varphi_i)(1-2\cosh(r)) & 2\cos(\varphi_i)(\tanh(\frac{r}{2})-\sinh(r))
& \sin^2(\varphi_i)\sinh(r)+\cos^2(\varphi_i)\tanh(\frac{r}{2})
\end{pmatrix}.
\]}

Using the expression $\mathcal{S}^r=E(r)C(r)^{-1}$ and some tedious
but elementary calculations we get to the main results of this
section.

\begin{proposition}
The shape operator $\mathcal{S}^r$ of the tube $M^r$ around the
homogeneous submanifold $S_{\g{w}}$ of $AN$ with respect to the
decomposition
$T_{\gamma(r)}M^r=\R(\sech(\frac{r}{2})B+\tanh(\frac{r}{2})\xi)
\oplus\g{U}\oplus\g{H} \oplus(\oplus_{i=1}^{\mi-1}\g{F}_i)
\oplus(\oplus_{i=\mi}^{\mii}\g{M}_i)\oplus(\oplus_{i=\mii+1}^{m}\g{P}_i)$
is given by
\begin{align*}
\mathcal{S}^r={}&\Bigl(\frac{1}{2}\tanh\Bigl(\frac{r}{2}\Bigr)\id_{l+1}\Bigr)
\oplus\Bigl(\frac{1}{2}\coth\Bigl(\frac{r}{2}\Bigr)\id_{h}\Bigr)
\oplus\left(\bigoplus_{i=1}^{\mi-1}
\begin{pmatrix}
\frac{1}{2}\coth(\frac{r}{2}) &-\frac{1}{2}\sech(\frac{r}{2})\\
-\frac{1}{2}\sech(\frac{r}{2}) &\tanh(\frac{r}{2})
\end{pmatrix}\right)\\
&{}\oplus\left(\bigoplus_{i=\mi}^{\mii}
\begin{pmatrix}
\frac{1}{2}\tanh(\frac{r}{2}) & 0
& -\frac{1}{2}\sin(\varphi_i)\sech(\frac{r}{2})\\[0.5ex]
0 & \frac{1}{2}\coth(\frac{r}{2})
& -\frac{1}{2}\cos(\varphi_i)\sech(\frac{r}{2})\\[0.5ex]
-\frac{1}{2}\sin(\varphi_i)\sech(\frac{r}{2})
& -\frac{1}{2}\cos(\varphi_i)\sech(\frac{r}{2}) & \tanh(\frac{r}{2})
\end{pmatrix}\right)\\
&{}\oplus\left(\bigoplus_{i=\mii+1}^{m}
\begin{pmatrix}
\frac{1}{2}\tanh(\frac{r}{2}) & -\frac{1}{2}\sech(\frac{r}{2})\\[0.5ex]
-\frac{1}{2}\sech(\frac{r}{2})& \tanh(\frac{r}{2})
\end{pmatrix}\right).
\end{align*}
\end{proposition}

As a consequence, we immediately get

\begin{corollary}
The mean curvature $\mathcal{H}^r$ of the tube $M^r$ at the point
$\gamma(r)$ is
\begin{align*}
\mathcal{H}^r(\gamma(r))&=\frac{1}{2}\left((h+\mii)\coth\left(\frac{r}{2}\right)
+(2+l+3m-\mi)\tanh\left(\frac{r}{2}\right)\right)\\
& = \frac{1}{2}\left((\codim S_\g{w}-1)\coth\left(\frac{r}{2}\right)
+(\dim S_\g{w}+\dim \g{z})\tanh\left(\frac{r}{2}\right)\right).
\end{align*}
Therefore, for every $r>0$, the tube $M^r$ around $S_\g{w}$ is a
hypersurface with constant mean curvature, and hence, tubes around
the submanifold $S_\g{w}$ constitute an isoparametric family of
hypersurfaces in $AN$, that is, every tube $M^r$ is an isoparametric
hypersurface.
\end{corollary}

We can also give the characteristic polynomial of $\mathcal{S}^r$,
which can be written as
\[
p_{r,\xi}(x)=(\lambda-x)^{l+1}\left(\frac{1}{4\lambda}-x\right)^{\!h}\,
\prod_{i=1}^{m} q^i_{r,\xi}(x),
\]
where $\lambda=\frac{1}{2}\tanh\left(\frac{r}{2}\right)$, and
\begin{align*}
q^i_{r,\xi}(x)= {}& x^2-\left(2\lambda+\frac{1}{4\lambda}\right)x
+\frac{1}{4}+\lambda^2,\quad \text{ if } i=1,\dots, \mi-1,\\
q^i_{r,\xi}(x)= {}& -x^3+\left(3\lambda+\frac{1}{4\lambda}\right)x^2
-\frac{1}{2}\left(6\lambda^2+1\right)x
+\frac{16\lambda^4+16\lambda^2-1+(4\lambda^2-1)^2\cos 2\varphi_i}{32\lambda},
\\ {}& \text{ if } i=\mi,\dots,\mii,\\
q^i_{r,\xi}(x)= {}& x^2-3\lambda x-\frac{1}{4}+3\lambda^2,\quad \text{ if } i=\mii+1,\dots, m.
\end{align*}
The zeroes of $p_{r,\xi}$ are the principal curvatures of the tube
$M^r$ at the point  $\gamma(r)$. Notice that the zeroes of
$q^i_{r,\xi}$ for $i\in\{1,\dots,\mi-1\}$ are
$\lambda=\frac{1}{2}\tanh\left(\frac{r}{2}\right)$ and
$\lambda+\frac{1}{4\lambda}=\coth(r)$, while for
$i\in\{\mii+1,\dots,m\}$ they are
$\frac{1}{2}\left(3\lambda\pm\sqrt{1-3\lambda^2}\right)$. If
$i\in\{\mi,\dots,\mii\}$ the zeroes of $q^i_{r,\xi}$ are given by
complicated expressions, because they are solutions of a cubic
polynomial. This polynomial coincides with the one in~\cite[p.\
146]{BD09} (where an analysis of its zeroes is carried out) and
in~\cite[p.\ 5]{DRDV10}.

From these results we deduce that, in general, the principal
curvatures of $M^r$,  and even the number of principal curvatures of
$M^r$, may vary from point to point, which implies that in general,
$M^r$ is an inhomogeneous hypersurface. Actually, \emph{the principal
curvatures of $M^r$ are constant if and only if $\g{w}^\perp$ has
constant generalized K\"{a}hler angle}, that is, if the $m$-tuple
$(\varphi_1,\dots,\varphi_m)$ does not depend on $\xi$.

We summarize the main results obtained so far.

\begin{theorem}\label{thMain}
Let $AN$ be a Damek-Ricci space with Lie algebra $\g{a}\oplus\g{n}$,
where $\g{a}$ is one-dimensional and $\g{n}=\g{v}\oplus\g{z}$ is a
generalized Heisenberg algebra with center $\g{z}$. Let $S_{\g{w}}$
be the connected subgroup of $AN$ whose Lie algebra is
$\g{s}_{\g{w}}=\g{a}\oplus\g{w}\oplus\g{z}$, where $\g{w}$ is any
proper subspace of $\g{v}$.

Then, the tubes around the submanifold $S_{\g{w}}$ are isoparametric
hypersurfaces of $AN$, and have constant principal curvatures if and
only if $\g{w}^\perp=\g{v}\ominus\g{w}$ has constant generalized
K\"{a}hler angle.
\end{theorem}


\section{Rank-one symmetric spaces of noncompact type}\label{secSymmetric}

In this section we present some particular examples of isoparametric
hypersurfaces in  the noncompact rank one symmetric spaces of
nonconstant curvature. Note that, in the case of real hyperbolic
spaces, our method only gives rise to tubes around totally geodesic
real hyperbolic subspaces, which are well-known examples.

\subsection{Complex hyperbolic spaces $\C H^n$}\label{sec:CHn}
The study of this case was the aim of \cite{DRDV10}. As explained
there,  isoparametric hypersurfaces arising from our method are
homogeneous if and only if they have constant principal curvatures
(and hence, if and only if $\g{w}^\perp$ has constant K\"ahler
angle). It follows that for every $n\geq 3$ one obtains inhomogeneous
isoparametric hypersurfaces with nonconstant principal curvatures in
$\C H^n$.

\subsection{Quaternionic hyperbolic spaces $\mathbb{H}
H^n$}\label{sec:HHn}

Our definition of generalized K\"ahler angle includes as a particular
case the notion of quaternionic K\"ahler angle introduced in
\cite{BB01}. The construction of several examples of subspaces of
$\H^{n-1}$ with constant quaternionic K\"ahler angle led Berndt and
Br\"uck to some examples of cohomogeneity one actions on $\H H^n$. In
\cite{BT07}, Berndt and Tamaru proved that these examples exhaust all
cohomogeneity one actions on $\H H^n$ with a non-totally geodesic
singular orbit whenever $n=2$ or the codimension of the singular
orbit is two.

Moreover, they reduced the problem of classifying cohomogeneity one
actions on $\H H^n$ to the following one:  \emph{find all subspaces
$\g{w}^\perp$ of $\g{v}=\H^{n-1}$ with constant quaternionic K\"ahler
angle and determine for which of them there exists a subgroup of
$Sp(n-1)Sp(1)$ that acts transitively on the unit sphere of
$\g{w}^\perp$ (via the standard representation on $\H^{n-1}$)}.
However, a complete classification of cohomogeneity one actions on
quaternionic hyperbolic spaces is not yet known, and neither is a
classification of the subspaces of $\H^{n-1}$ with constant
quaternionic K\"ahler angle, which seems to be a difficult linear
algebra problem. Furthermore, it is not clear whether an answer to
this latter problem would directly lead to the answer of the former.
In fact, in view of Theorem \ref{thMain} a subspace $\g{w}^\perp$ of
$\g{v}$ with constant quaternionic K\"ahler angle gives rise to an
isoparametric hypersurface in $\H H^n$ with constant principal
curvatures, but then one would have to decide whether this
hypersurface is homogeneous or not. Nonetheless, what Theorem
\ref{thMain} guarantees, as well as in the case of complex hyperbolic
spaces, is the existence of inhomogeneous isoparametric hypersurfaces
with nonconstant principal curvatures in $\H H^n$, for every $n\geq
3$.

The subspaces of $\H^{n-1}$ with constant quaternionic K\"ahler angle
known up to now  can take the following values of quaternionic
K\"ahler angles \cite{BT07}: $(0,0,0)$, $(0,\pi/2,\pi/2)$,
$(\pi/2,\pi/2,\pi/2)$, $(0,0,\pi/2)$, $(\varphi,\pi/2,\pi/2)$ and
$(0,\varphi,\varphi)$. In this subsection, we will give new examples
of subspaces of $\mathbb{H}^{n-1}$, $n\geq 5$, with constant
quaternionic K\"ahler angle $(\varphi_1,\varphi_2,\varphi_3)$, with
$0<\varphi_1\leq \varphi_2\leq\varphi_3\leq\pi/2$, and
$\cos(\varphi_1)+\cos(\varphi_2)<1+\cos(\varphi_3)$. This includes,
for example, the cases $(\varphi,\varphi,\varphi)$, with
$0<\varphi<\pi/2$, and $(\varphi_1,\varphi_2,\pi/2)$, with
$\cos(\varphi_1)+\cos(\varphi_2)<1$. Theorem \ref{thMain} ensures
that these new subspaces yield new examples of isoparametric
hypersurfaces with constant principal curvatures in $\H H^n$. In
fact, these hypersurfaces are homogeneous, as shown in Theorem
\ref{th:cohom1}. This provides a large new family of cohomogeneity
one actions on quaternionic hyperbolic spaces.

From now on, $(i,i+1,i+2)$ will always be a cyclic permutation of
$(1,2,3)$. Fix a canonical basis $\{J_1,J_2,J_3\}$ of the
quaternionic structure of $\H^{n-1}$, that is, $J_i^2=-\id$ and $J_i
J_{i+1}=J_{i+2}=-J_{i+1}J_i$, with $i\in\{1,2,3\}$.

Let $0<\varphi_1\leq \varphi_2\leq\varphi_3\leq\pi/2$ with
$\cos(\varphi_1)+\cos(\varphi_2)<1+\cos(\varphi_3)$, and consider a
four dimensional totally real subspace of $\H^{n-1}$ and a basis of
unit vectors $\{e_0,e_1,e_2,e_3\}$ of it, where $\langle
e_0,e_i\rangle=0$, for $i=1,2,3$, and
\begin{equation}\label{eq:inner_products}
\langle e_{i}, e_{i+1}\rangle
=\frac{\cos(\varphi_{i+2})-\cos(\varphi_i)\cos(\varphi_{i+1})}
{\sin(\varphi_i)\sin(\varphi_{i+1})},\quad i=1,2,3.
\end{equation}
The existence of these vectors is ensured by the following

\begin{lemma}
We have:
\begin{enumerate}[{\rm (a)}]
\item Let $\alpha_1$, $\alpha_2$, $\alpha_3\in\R$. Then, there
    exists a basis of unit vectors $\{e_1,e_2,e_3\}$ of $\R^3$
    such that $\langle e_i,e_{i+1}\rangle=\alpha_{i+2}$ if and
    only if $\lvert\alpha_i\rvert<1$ for all $i$, and
    $\alpha_1^2+\alpha_2^2+\alpha_3^2<1+2\alpha_1\alpha_2\alpha_3$.\label{th:ei:1}
\item Assume $0<\varphi_1\leq \varphi_2\leq\varphi_3\leq\pi/2$.
    Then, there exists a basis of unit vectors $\{e_1,e_2,e_3\}$
    of $\R^3$ with the inner products as in
    \eqref{eq:inner_products} if and only if
    $\cos(\varphi_1)+\cos(\varphi_2)<1+\cos(\varphi_3)$.\label{th:ei:2}
\end{enumerate}
\end{lemma}
\begin{proof}
The proof of (\ref{th:ei:1}) is elementary so we omit it. For the
proof of (\ref{th:ei:2}) we only give a few basic indications and
leave the details to the reader.

We define $x_i=\cos(\varphi_i)$. With this notation, the conditions
$\lvert\alpha_i\rvert<1$ in part (\ref{th:ei:1}) are equivalent to
$x_1^2+x_2^2+x_3^2<1+2x_1 x_2 x_3$, whereas the condition
$\alpha_1^2+\alpha_2^2+\alpha_3^2<1+2\alpha_1\alpha_2\alpha_3$ turns
out to be equivalent to
\[
\frac{\left(x_1-x_2-x_3+1\right) \left(x_1+x_2-x_3-1\right) \left(x_1-x_2+x_3-1\right)
\left(x_1+x_2+x_3+1\right)}{\left(x_1^2-1\right) \left(x_2^2-1\right)
   \left(x_3^2-1\right)}<0.
\]
Since $0<\varphi_1\leq \varphi_2\leq\varphi_3\leq\pi/2$, we have
$0\leq x_3\leq x_2\leq x_1<1$, and thus the equation above is
equivalent to $x_1+x_2-x_3-1<0$. Finally, it is not hard to show that
$0\leq x_3\leq x_2\leq x_1<1$, and $x_1+x_2-x_3-1<0$ imply
$x_1^2+x_2^2+x_3^2<1+2x_1 x_2 x_3$.
\end{proof}

For the sake of simplicity let us define $\varphi_0=0$ and $J_0=\id$.
Notice that $\langle J_j e_k, e_l\rangle=0$ for $j\in\{1,2,3\}$ and
$k$, $l\in\{0,1,2,3\}$, because $\mathrm{span}\{e_0,e_1,e_2,e_3\}$ is
a totally real subspace of $\H^{n-1}$. Then we can define
\[
\xi_k = \cos(\varphi_k)J_k e_0+\sin(\varphi_k)J_k e_k, \quad k\in\{0,1,2,3\}.
\]
(Note that $\xi_0=e_0$.) We consider the subspace $\g{w}^\perp$
generated by these four vectors, for which $\{\xi_0, \xi_1,
\xi_2,\xi_3\}$ is an orthonormal basis. Now, taking a generic unit
vector $\xi=a_0\xi_0+a_1 \xi_1+a_2\xi_2+a_3\xi_3\in \g{w}^\perp$,
some straightforward calculations show that the matrix of the
quadratic form $Q_\xi$ defined in Theorem~\ref{thGKA} with respect to
the basis $\{J_1,J_2,J_3\}$ (i.e. the matrix whose entries are
$\langle F_j \xi, F_k \xi\rangle=\sum_{l=0}^3\langle J_j
\xi,\xi_l\rangle\langle J_k \xi,\xi_l\rangle$, for $j$,
$k\in\{1,2,3\}$) is the diagonal matrix with entries
$\cos^2(\varphi_1)$, $\cos^2(\varphi_2)$, $\cos^2(\varphi_3)$.
Therefore, {$\g{w}^\perp$ has constant quaternionic K\"ahler angle
$(\varphi_1,\varphi_2,\varphi_3)$}.

Next, we show that the submanifold $S_\g{w}$ is the singular orbit of
a cohomogeneity one action on $\H H^n$, and hence, the tubes around
$S_\g{w}$ are homogeneous isoparametric hypersurfaces. Let
$G=Sp(n,1)$ be the connected component of the identity of the
isometry group of $\H H^n$, and let $K=Sp(n)Sp(1)$ be the isotropy
group of $G$ at the identity element of $AN=\H H^n$. Denote by
$N_K(S_\g{w})=\{k\in K : kS_\g{w}k^{-1}\subset S_\g{w}\}$ the
normalizer of $S_\g{w}$ in $K$, and by $N_K^0(S_\g{w})$ the connected
component of the identity. Notice that $S_\g{w}$ can be seen as a
submanifold of $AN=\H H^n$, and also as a subgroup of $AN\subset G$.
We have:

\begin{theorem}\label{th:cohom1}
Let $\g{w}^\perp$ be the subspace of $\g{v}=\H^{n-1}$ of constant
quaternionic K\"ahler angle $(\varphi_1,\varphi_2,\varphi_3)$, with
$0<\varphi_1\leq \varphi_2\leq\varphi_3\leq\pi/2$, and
$\cos(\varphi_1)+\cos(\varphi_2)<1+\cos(\varphi_3)$, as defined
above, and consider $\g{w}=\g{v}\ominus\g{w}^\perp$. Then:
\begin{enumerate}[{\rm (a)}]
\item The tubes around the submanifold $S_\g{w}$ are
    isoparametric hypersurfaces with constant principal
    curvatures.\label{th:cohom1:known}

\item There is a subgroup of $Sp(n-1)Sp(1)$ that acts
    transitively on the unit sphere of $\g{w}^\perp$ (via the
    standard representation on
    $\g{v}=\H^{n-1}$).\label{th:cohom1:a}

\item The subgroup $N_{K}^0(S_\g{w})S_{\g{w}}$ of $G$ acts
    isometrically with cohomogeneity one on $\H H^n$, $S_\g{w}$
    is a singular orbit of this action, and the other orbits are
    tubes around $S_\g{w}$.\label{th:cohom1:b}
\end{enumerate}
\end{theorem}

\begin{proof}
Assertion (\ref{th:cohom1:known}) follows from previous calculations
in this section. Part (\ref{th:cohom1:b}) follows from part
(\ref{th:cohom1:a}) using \cite[p.~220]{BB01}
(cf.~\cite[Theorem~4.1(i)]{BT07}). Let us then prove
(\ref{th:cohom1:a}). For the case $(\varphi,\pi/2,\pi/2)$, with
$0<\varphi\leq\pi/2$, assertion (\ref{th:ei:2}) is already known to
be true~\cite{BB01}. If $0<\varphi_1\leq\varphi_2<\varphi_3=\pi/2$,
then the proof given below needs to be adapted: we would get $F_3=0$,
and hence $\bar{F}_3$ would not be defined; in that case we would
explicitly define $\bar{F}_3=\bar{F}_1\bar{F}_2$. Thus, we will
assume $\varphi_3<\pi/2$ in what follows.

Let $i\in\{1,2,3\}$. As above, henceforth $(i,i+1,i+2)$ will be a
cyclic permutation of $(1,2,3)$. First note that ${F}_i\xi_0=\langle
J_i\xi_0,\xi_i\rangle \xi_i=\cos(\varphi_i)\xi_i$, and thus
$\bar{F}_i\xi_0=\xi_i$. This implies $\langle
F_i\xi_{i+1},\xi_0\rangle = -\langle \xi_{i+1},F_i\xi_0\rangle =
-\cos(\varphi_i)\langle\xi_{i+1},\xi_i\rangle =0$. The skew-symmetry
of $J_i$ yields $\langle F_i\xi_{i+1},\xi_{i+1}\rangle=0$, and from
$J_i\xi_{i+1}=\cos(\varphi_{i+1})J_{i+2}e_0+\sin(\varphi_{i+1})J_{i+2}e_{i+1}$
we get $\langle F_i\xi_{i+1},\xi_i\rangle =0$ using $\langle J_j
e_k,e_l\rangle =0$. Altogether this implies that $F_i\xi_{i+1}$ must
be a multiple of $\xi_{i+2}$, and hence one readily gets
$\bar{F}_i\xi_{i+1}=\xi_{i+2}=-\bar{F}_{i+1}\xi_i$. Applying these
results twice, we get
$\bar{F}_i\bar{F}_{i+1}=\bar{F}_{i+2}=-\bar{F}_{i+1}\bar{F}_i$, and
$\bar{F}_i^2=-\id$, so $\{\bar{F}_1,\bar{F}_2,\bar{F}_3\}$ is a
quaternionic structure on $\g{w}^\perp$.

Let $\eta_0\in\g{w}^\perp$ be an arbitrary unit vector. We define
$f_0=\eta_0$ and apply Lemma~\ref{lemma:BB} to find unit vectors
$f_1$, $f_2$, $f_3\in\H^{n-1}$ orthogonal to $f_0$, such that
$\eta_i=\bar{F}_i f_0=\cos(\varphi_i)J_i f_0+\sin(\varphi_i)J_i f_i$.
Then $f_i=-(J_i\bar{F}_i f_0+\cos(\varphi_i) f_0)/\sin(\varphi_i)$.
We easily obtain $\langle J_i\bar{F}_i f_0,f_0\rangle
=-\langle\bar{F}_i f_0,F_i f_0\rangle = -\cos(\varphi_i)$, and using
$\bar{F}_{i+2}\bar{F}_{i+1}=-\bar{F}_{i}$ we get
\begin{align*}
\langle J_i\bar{F}_i f_0,J_{i+1}\bar{F}_{i+1}f_0\rangle
&=-\langle\bar{F}_i f_0, J_{i+2}\bar{F}_{i+1}f_0\rangle
= -\cos(\varphi_{i+2})\langle \bar{F}_i f_0,\bar{F}_{i+2}\bar{F}_{i+1}f_0\rangle\\
&=\cos(\varphi_{i+2})\langle\bar{F}_i f_0,\bar{F}_i f_0\rangle=\cos(\varphi_{i+2}).
\end{align*}
Altogether this implies
\[
\langle f_i,f_{i+1}\rangle=\frac{\cos(\varphi_{i+2})-\cos(\varphi_i)\cos(\varphi_{i+1})}
{\sin(\varphi_i)\sin(\varphi_{i+1})}=\langle e_i,e_{i+1}\rangle.
\]
We show that $\langle J_j f_k,f_l\rangle =0$ for all $j\in\{1,2,3\}$,
and $k$, $l\in\{0,1,2,3\}$. For example,
\begin{align*}
\langle J_i f_{i+1},f_{i+2}\rangle
&=\frac{1}{\sin(\varphi_{i+1})\sin(\varphi_{i+2})}
\langle J_{i+2}\bar{F}_{i+1}f_0+\cos(\varphi_{i+1})J_i f_0,
    J_{i+2}\bar{F}_{i+2}f_0+\cos(\varphi_{i+2})f_0\rangle.
\end{align*}
Using the properties of the generalized K\"{a}hler angle (see
Theorem~\ref{thGKA}), and the definition of $\bar{F}_i$, we obtain
$\langle J_{i+2}\bar{F}_{i+1}f_0,J_{i+2}\bar{F}_{i+2}f_0\rangle
=\langle \bar{F}_{i+1}f_0,\bar{F}_{i+2}f_0\rangle = 0$. Similarly,
one gets $\langle J_{i+2}\bar{F}_{i+1} f_0,f_0\rangle =\langle
J_{i}f_0,J_{i+2}\bar{F}_{i+2}f_0\rangle=\langle
J_{i}f_0,f_0\rangle=0$, and hence $\langle J_i
f_{i+1},f_{i+2}\rangle=0$. Other combinations of indices can be
handled analogously to obtain $\langle J_j f_k,f_l\rangle =0$.

Now, one can apply the Gram-Schmidt process to $\{e_0,e_1,e_2,e_3\}$
to obtain an $\H$-ortho\-nor\-mal set $\{e_0',e_1',e_2',e_3'\}$, and
similarly with $\{f_0,f_1,f_2,f_3\}$, to obtain an $\H$-orthonormal
set $\{f_0',f_1',f_2',f_3'\}$. Then there exists an element $T\in
Sp(4)\subset Sp(n-1)\subset Sp(n-1)Sp(1)$ such that $Te_i'=f_i'$ for
$i=0,1,2,3$, and $TJ_l=J_lT$ for $l=1,2,3$. Since the transition
matrices from $\{e_i'\}$ to $\{e_i\}$, and from $\{f_i'\}$ to
$\{f_i\}$ coincide, we get $Te_i=f_i$, and hence $T\xi_i=\eta_i$ for
$i=0,1,2,3$. Therefore, $T\xi_0=\eta_0$ and
$T\g{w}^\perp=\g{w}^\perp$. Since $\eta_0\in\g{w}^\perp$ is
arbitrary, (\ref{th:cohom1:a}) follows.
\end{proof}

\begin{remark}
This construction can be extended to subspaces of $\H^{n-1}$ (for $n$
sufficiently high) with real dimension multiple of four and with
constant quaternionic K\"ahler angle
$(\varphi_1,\varphi_2,\varphi_3)$ as before, just by considering
orthogonal sums of subspaces $\g{w}^\perp$ like the one constructed
above. Theorem~\ref{th:cohom1} can easily be extended to show the
homogeneity of the corresponding isoparametric hypersurfaces.
\end{remark}

\subsection{The Cayley hyperbolic plane
$\mathbb{O}H^2$}\label{sec:OH2}

Based on some results of \cite{BB01}, Berndt and Tamaru achieved the
classification of homogeneous hypersurfaces in the Cayley hyperbolic
plane $\mathbb{O}H^2$ \cite{BT07}. Some of these homogeneous examples
appear as particular cases of the construction we have developed. If
we put $k=\dim \g{w}^\perp$, then for $k\in\{1,2,3,4,6,7,8\}$ the
tubes $M^r$ around $S_\g{w}$ are homogeneous hypersurfaces for every
$r>0$ and, together with $S_\g{w}$, constitute the orbits of a
cohomogeneity one action; if $k=5$, none of the tubes around
$S_\g{w}$ is homogeneous~\cite[p.~233]{BB01}.

Therefore, our method yields a family of {inhomogeneous isoparametric
hypersurfaces} if and only if the codimension of $S_{\g{w}}$ is
$k=5$. But for this case, something unexpected happens: these
hypersurfaces have {constant principal curvatures}.

Let $\xi$ be a unit vector in $\g{w}^\perp$. Taking into account that
$\g{v}=\R^8$ is an irreducible Clifford module of $\g{z}=\R^7$, the
properties of generalized Heisenberg algebras imply that the linear
map $Z\in\g{z}\mapsto J_Z\xi\in\g{v}\ominus\R \xi$ is an isometry.
Hence, we can find an orthonormal basis $\{Z_1,\dots, Z_7\}$ of the
vector space $\g{z}$ such that $\g{w}=\textrm{span}\{J_{Z_5} \xi,
J_{Z_6} \xi, J_{Z_7} \xi\}$ and $\g{w}^\perp=\textrm{span}\{\xi,
J_{Z_1} \xi, J_{Z_2} \xi, J_{Z_3} \xi, J_{Z_4} \xi\}$. It is then
clear that $J_{Z_5}, J_{Z_6}, J_{Z_7}$ map $\xi$ into $\g{w}$ and
$J_{Z_1}, J_{Z_2}, J_{Z_3}, J_{Z_4}$ map $\xi$ into $\g{w}^\perp$. By
definition, the generalized K\"ahler angle of $\xi$ with respect to
$\g{w}^\perp$ is $(0,0,0,0,\pi/2,\pi/2,\pi/2)$.

As the above argument is valid for every unit $\xi\in\g{w}^\perp$, we
conclude that $\g{w}^\perp$ has constant K\"ahler angle and so, by
Theorem \ref{thMain} and the inhomogeneity result in \cite{BB01}, we
obtain:

\begin{theorem}\label{th:CH2}
The tubes around the homogeneous submanifolds $S_\g{w}$ with
$\dim\g{w}^\perp=5$ are inhomogeneous isoparametric hypersurfaces
with constant principal curvatures in the Cayley hyperbolic plane.
\end{theorem}

Let us consider now the case $k=4$. A similar argument as above can
be used to show that the generalized K\"{a}hler angle of $\g{w}^\perp$ is
$(0,0,0,\pi/2,\pi/2,\pi/2,\pi/2)$. Then, the calculations before
Theorem~\ref{thMain} show that for any choice of $\g{w}^\perp$ the
principal curvatures, and their corresponding multiplicities, of the
tube of radius $r$ around $S_{\g{w}}$, depend only on $r$.
By~\cite[Theorem 4.7]{BT07}, it follows that there are uncountably
many orbit equivalence classes of cohomogeneity one actions on
$\mathbb{O} H^2$ arising from this method with $k=4$. Therefore, we
obtain an \emph{uncountable family of noncongruent homogeneous
isoparametric systems with the same constant principal curvatures,
counted with multiplicities}. This phenomenon was known in the
inhomogeneous case for spheres~\cite{FKM81}, and in the homogeneous
case for noncompact symmetric spaces of rank higher than
two~\cite{BT03}.


\end{document}